\documentclass[12pt,a4paper,reqno]{amsart}
\usepackage{amsmath}
\allowdisplaybreaks
\usepackage{amsfonts}
\usepackage{amssymb,amsthm,amsfonts,amsthm,latexsym,enumerate,url,cases}
\numberwithin{equation}{section}
\usepackage{relsize}
\usepackage{mathrsfs}
\usepackage{hyperref}
\hypersetup{colorlinks=true,citecolor=blue,linkcolor=blue,urlcolor=blue}
\usepackage{extarrows}
\usepackage{tikz,xcolor,hyperref}

\definecolor{lime}{HTML}{A6CE39}
\DeclareRobustCommand{\orcidicon}{
    \begin{tikzpicture}
    \draw[lime, fill=lime] (0,0)
    circle [radius=0.16]
    node[white] {{\fontfamily{qag}\selectfont \tiny ID}};
    \draw[white, fill=white] (-0.0625,0.095)
    circle [radius=0.007];
    \end{tikzpicture}
    \hspace{-2mm}
}
\foreach \x in {A, ..., Z}{\expandafter\xdef\csname orcid\x\endcsname{\noexpand\href{https://orcid.org/\csname orcidauthor\x\endcsname}
            {\noexpand\orcidicon}}
}

\def\pmod #1{\ ({\rm{mod}}\ #1)}

\newtheorem{theorem*}{Theorem}
\newtheorem{lemma*}{Lemma}
\theoremstyle{plain}
\newtheorem{Hypothesis}{Hypothesis}
\newtheorem{theorem}{Theorem}
\newtheorem{lemma}[theorem]{Lemma}
\newtheorem{corollary}[theorem]{Corollary}
\newtheorem{proposition}{Proposition}
\newtheorem{conjecture}{Conjecture}
\theoremstyle{definition}
\newtheorem*{acknowledgment}{Acknowledgments}
\newtheorem{remark}{Remark}


\begin{document}

\title
[{Quantitative results of the Romanov type representation functions}] {Quantitative results of the Romanov type representation functions}

\author
[Yong-Gao Chen and Yuchen Ding] {Yong-Gao Chen\orcidA{} and Yuchen Ding*\orcidB{}}

\address{(Yong-Gao Chen) School of Mathematical Sciences,
Nanjing Normal University, Nanjing 210023, People's Republic of
China} \email{ygchen@njnu.edu.cn}
\address{(Yuchen Ding) School of Mathematical Sciences,  Yangzhou University, Yangzhou 225002, People's Republic of China}
\email{ycding@yzu.edu.cn}
\thanks{*Corresponding author}


\begin{abstract} For $\alpha >0$, let $$\mathscr{A}=\{ a_1<a_2<a_3<\cdots\}$$ and $$\mathscr{L}=\{ \ell_1, \ell_2, \ell_3,\cdots\} \quad \text{(not~necessarily~different)}$$ be two sequences of positive
integers with $\mathscr{A}(m)>(\log m)^\alpha $ for infinitely
many positive integers $m$ and $\ell_m<0.9\log\log m$ for sufficiently large integers $m$.
Suppose further that $(\ell_i,a_i)=1$ for all $i$. For any $n$, let $f_{\mathscr{A},\mathscr{L}}(n)$ be the number of the available representations listed below
$$\ell_in=p+a_i \quad \left(1\le i\le \mathscr{A}(n)\right),$$
where $p$ is a prime number.
It is proved that
$$\limsup_{n\to \infty } \frac{f_{\mathscr{A},\mathscr{L}}(n)}{\log\log n}>0,$$
which covers an old result of Erd\H os in 1950 by taking $a_i=2^i$ and $\ell_i=1$. One key ingredient in the argument is a technical lemma established here which illustrates how to pick out the admissible parts of an arbitrarily given set of distinct linear functions. The proof then reduces to the verifications of a hypothesis involving well--distributed sets introduced by Maynard, which of course would be the other key ingredient in the argument.

\end{abstract}
\maketitle

\baselineskip 18pt

\section{Introduction}
In 1849, de Polignac \cite{de1} conjectured that every odd number greater than 3 is the sum of a prime and a power of 2.
 However, he \cite{de2} found that 127 and 959 are two counterexamples soon. In 1934, Romanov \cite{Ro} proved that there
 is a positive proportion of the odd numbers which can be represented by the form $p+2^k~~(p\in \mathbb{P},~k\in \mathbb{N})$,
  where $\mathbb{P}$ and $\mathbb{N}$ stand for the set of primes and natural numbers, respectively.
  Later, the opposite direction was achieved by van der Corput \cite{va}, who illustrated that the odd numbers
  which can not be written as the form $p+2^k~~(p\in \mathbb{P},~k\in \mathbb{N})$ still possess a positive lower density.
  In the same year, Erd\H os \cite{Er1} constructed an arithmetic progression of odd numbers none of whose elements can be
  represented by the sum of a prime and a power of 2. It is easy to find that van der Corput's result follows from the one of Erd\H os.

In 2004, the first author and Sun \cite{CS} consider the quantitative version of the Romanov theorem. They showed that the lower asymptotic density of the odd numbers represented by the form $p+2^k~~(p\in \mathbb{P},~k\in \mathbb{N})$ is larger than $0.0868$. In their paper, they pointed out that Erd\H os' construction leads to the fact that the upper asymptotic density of the odd numbers with those forms is less than $0.49999991$. There are a few improvements of the lower asymptotic density, such as \cite{ES,HL,HS,Lv,Pi}. As far as we know, the best numerical values of upper asymptotic density at present is $0.49095$ due to  Habsieger and Roblot \cite{HL}. In literatures, the best numerical value of the lower asymptotic density is $0.107648$ obtained by Elsholtz and Schlage--Puchta \cite{ES}. However, by a review of Alessandro regarding the paper of Elsholtz and Schlage--Puchta \cite{ES} on MathSciNet, the lower asymptotic density is at least $0.110114$.

Recently, Del Corso et al. \cite{DDD} thought that if there is some absolute constant $c_0$ such that
$$\lim_{x\rightarrow\infty}\frac{\#\{n:n\leqslant x,~ n=p+2^m, ~m\in\mathbb{N},~p\in\mathbb{P}\}}{x}=c_0,$$
then $0.437...$ is probably a candidate of it. Their calculations follow the heuristic argument of Bombieri as well as a large amount of numeric operations.

Define the representation function $f(n)$ to be $$f(n)=\#\left\{(p,k):n=p+2^k,~p\in \mathbb{P},~k\in \mathbb{N}\right\}.$$
Noting that (in the proof of the Romanov theorem)$$\sum_{n\leqslant x}f^2(n)\ll x,$$
Tur\'{a}n asked Erd\H os (written communication) whether the function $f(n)$ is the unbounded one.
Erd\H os \cite{Er1} gave an affirmative answer to the question of Tur\'{a}n by showing the following stronger result:
\begin{align}\label{eq0}
\limsup_{n\rightarrow\infty}\frac{f(n)}{\log\log n}>0.
\end{align}
The unboundedness of the function $f(n)$
for integers $n$ having at most two prime factors was then established by Lu \cite{Lu},
following a remark of Friedlander and Iwaniec \cite{FI} which claimed that $f(n)$ is greater than 2
for infinitely many integers $n$ with at most two prime factors.
Further, for any subset $\mathscr{A}$ of $\mathbb{N}$, let
$$f_{2^\mathscr{A}}(n)=\#\left\{(p,a):n=p+2^a,~p\in \mathbb{P},~a\in \mathscr{A}\right\}.$$
With this notation, we shall have $f(n)=f_{2^\mathbb{N}}(n)$.
Erd\H os' proof can be extended to the result
$$\limsup_{n\rightarrow\infty}\frac{f_{2^\mathscr{A}}(n)\log n}{\mathscr{A}\left(\frac{\log n}{\log2}\right)\log\log n}>0,$$
where $\mathscr{A}(x)=|\mathscr{A}\bigcap[1,x]|$. So if one takes
the subset $\mathscr{A}$ to be the prime set $\mathbb{P}$, we just
miss unboundedness fact of the representation function
$f_{2^\mathbb{P}}(n)$. Based on this observation, the first author
of the present paper \cite{Ch} posed the following conjecture:
$$\limsup_{n\rightarrow\infty}f_{2^\mathbb{P}}(n)=\infty.$$
This conjecture was solved by the second author and Zhou \cite{DZ} earlier.

It turns out that the magnitude of the function $f_{2^\mathscr{A}}(n)$ is quite difficult to decide. For example,
Erd\H os \cite{Er1} conjectured $f(n)=o(\log n)$. However, as Erd\H os \cite{Er1} said, he can even not prove that not all the integers
$$n-2^k~~\left(1\leqslant k<\frac{\log n}{\log 2}\right)$$
are primes for sufficiently large numbers $n$. A careful checking of the prime table (up to $203775$)
by Erd\H{o}s indicated that the largest exceptional integer should be 105. For integers $n\leqslant x$,
Vaughan \cite{Va} proved that the number of the exceptional integers is $O(x\exp(-c_1\log x\log\log\log x/\log\log x))$
for some constant $c_1>0$ by the Montgomery sieve \cite{Mo} (i.e., the arithmetic large sieve).
Under the assumption of the extended Riemann Hypothesis, Hooley \cite{Ho} proved that the exceptional numbers
are less than $O(x^{1-\lambda+\varepsilon})$ for any sufficiently small positive number $\varepsilon$,
where $\lambda=\prod_{p}\left(1-\frac{1}{p(p-1)}\right)$. Narkiewicz \cite{Na}
improved this to $O(x^{1-\lambda/\log 2+\varepsilon})$ under the same assumption.  Elsholtz \cite{Els} considered a related problem.

In the same paper, Erd\H{o}s made the following anecdotal conjecture.

{\bf Conjecture (Erd\H{o}s 1950).} {\it Let $c$ be any constant and $x$
sufficiently large,
$$a_1<a_2<\cdot\cdot\cdot<a_t\leqslant x,~t>\log x.$$
Then there exists an integer $n$ so that the number of solutions
of $n=p+a_i$ $(p\in \mathbb{P}, 1\le i\le t)$ is greater than
$c$.}

The case $a_i=2^i$ of the conjecture was proved by Erd\H{o}s himself as we mentioned above (Tur\'{a}n's question). In a former note
\cite{Di}, the second author gave a slight generalization of Erd\H os' theorem. To be precise,
suppose that $\mathscr{A}=\{a_1<a_2<a_3<\cdot\cdot\cdot\}$ is a set satisfying $\mathscr{A}(x)>\log x$ and
$a_i\mid a_{i+1}$ for all sufficiently large $x$ and $i$, then
\begin{align}\label{eq1}
\limsup_{n\rightarrow\infty}\frac{f_{\mathscr{A}}(n)}{\sqrt{\log\log n}}>0,
\end{align}
where $$f_{\mathscr{A}}(n)=\#\left\{(p,a):n=p+a,~p\in
\mathbb{P},~a\in \mathscr{A}\right\}.$$ After the second author
and Zhou \cite{DZ} proved the conjecture for the case
$a_i=2^{p_i}$, where $p_i$ is the $i$-th prime. The authors of the
present paper \cite{CD} recognized that the complete proof of
Erd\H{o}s' conjecture actually follows directly from a new
achievement of the distributions of the primes established by
Maynard--Tao \cite{Ma,Ta}. In fact, we proved the following more
stronger result.

{\bf Theorem A.} (Chen--Ding) {\it Let  $x\ge 2$ and the positive integers $a_i$ satisfy
$$a_1<a_2<\cdot\cdot\cdot<a_t\leqslant x,~t>\log x.$$
Then there exists infinitely many integers $n$ so that the number
of solutions of $n=p+a_i$ $(p\in \mathbb{P}, 1\le i\le t)$ is
greater than $ \frac 18 \log \log x -1.6$.}

In Theorem A, we do not know the magnitude of the least $n$. In
this paper, we give a quantitative result of the Erd\H os
conjecture.

Now, let us state the following theorem which is the main purpose of our paper.

\begin{theorem}\label{thm1}
For $0<\alpha <\frac 14$, there are two positive constants $c$ and $c'$ depending on $\alpha
$ such that if $x$ is a sufficiently large number, $t>(\log x)^\alpha$, $t$ pairs of positive integers $a_i,~\ell_i$ satisfy
$$a_1<a_2<\cdot\cdot\cdot<a_t\leqslant x,$$
and $$0<\ell_i<0.9\log\log x, \quad (\ell_i,a_i)=1 \quad (1\leqslant i\leqslant t),$$
then there are at least $x^{1/\alpha}\exp (-c (\log x)^\alpha )$
 integers $n\in \left[(x\log\log x)^{1/\alpha}, 2(x\log\log x)^{1/\alpha} \right)$ so that at least $c'\log\log n$ of the representations
$$\ell_in=p+a_i \quad (p\in \mathbb{P},1\leqslant i\leqslant t)$$
are available.
\end{theorem}

The quantitative result of the Erd\H os conjecture mentioned above is a special case of Theorem \ref{thm1} by taking $\ell_i=1$ for all $i$.

For two sequences $\mathscr{A}=\{ a_1<a_2<a_3<\cdots\}$ and $\mathscr{L}=\{ \ell_1, \ell_2, \ell_3,\cdots\}$ ($\ell_i$ not~necessarily~different) with $(\ell_i,a_i)=1$, let $f_{\mathscr{A},\mathscr{L}}(n)$ be the extended representation function defined as
$$f_{\mathscr{A},\mathscr{L}}(n)=\#\left\{i:\ell_i n=p+a_i,~a_i\in \mathscr{A},~\ell_i\in\mathscr{L},~p\in \mathbb{P},~1\le i\le \mathscr{A}(n)\right\}.$$
With this definition we shall have $f_{\mathscr{A},\mathscr{L}}(n)=f_{\mathscr{A}}(n)$ for $\ell_i=1$. From Theorem \ref{thm1}, we have the following corollaries.

\begin{corollary}\label{cor1}
For $\alpha >0$, let $$\mathscr{A}=\{ a_1<a_2<a_3<\cdots\}$$ and $$\mathscr{L}=\{ \ell_1, \ell_2, \ell_3,\cdots\} \quad \text{(not~necessarily~different)}$$ be two sequences of positive
integers with $\mathscr{A}(m)>(\log m)^\alpha $ for infinitely
many positive integers $m$ and $\ell_m<0.9\log\log m$ for sufficiently large integers $m$.  We have
$$\limsup_{n\to \infty } \frac{f_{\mathscr{A},\mathscr{L}}(n)}{\log\log n}>0.$$
\end{corollary}

\begin{remark} Taking $\mathscr{A}=\{2^i:i\in \mathbb{N}\}$ and $\ell_i=1$, then Corollary \ref{cor1} reduces to
$$\limsup_{\substack{n\to \infty }} \frac{f(n)}{\log\log n}>0,$$
which is the old result obtained by Erd\H os in 1950. Taking $\mathscr{A}=\{2^p:p\in \mathbb{P}\}$ and $\ell_i=1$, then Corollary \ref{cor1} reduces to
$$\limsup_{\substack{n\to \infty }} \frac{f_{2^{\mathbb{P}}}(n)}{\log\log n}>0,$$
which offers a much more stronger result obtained by the second author and Zhou formerly.
\end{remark}

 \begin{corollary}\label{cor1a} Let $\alpha >0$ and $\mathscr{A}=\{ a_1, a_2, \cdots   \} $ be a strictly increasing sequence of positive
integers with $\mathscr{A}(x)>(\log x)^\alpha $ for all
sufficiently large $x$. Then there are two positive constants $c$ and $c'$
depending only on $\alpha $ such that  for each
sufficiently large $x$, there are at least $x\exp (-c (\log
x)^\alpha )$ integers $n\le x$ with
$$f_{\mathscr{A}}(n)\ge c'\log\log x.$$
 \end{corollary}

\begin{corollary}\label{cor2}  Given an integer $\ell \ge 1$. Let
$f_\ell(n)$ be the number of solutions of $n=p+2^{m^\ell }$ with
$p\in \mathbb{P}$ and $m\in \mathbb{N}$. Then  there are two positive
constants $c$ and $c'$ depending only on $\ell $
 such that for each
sufficiently large $x$, there are at least $x\exp (-c (\log
x)^{1/\ell})$ integers $n\le x$ with
$$f_\ell (n)\ge c' \log\log x.$$
 \end{corollary}

\begin{remark} If $\ell \ge 2$, then
$$|\{ n\le x : f_\ell (n)\ge 1 \} | \ll \frac{x}{\log x} \cdot (\log
x)^{1/\ell } = \frac x{(\log x)^{1-1/\ell }}.$$ It follows that
the set of integers $n$ that can be represented as the form
$n=p+2^{m^\ell }$ with $p\in \mathbb{P}$ and $m\in \mathbb{N}$
has asymptotic density zero. But Corollary \ref{cor2} shows that
$$\limsup_{\substack{n\rightarrow\infty}} \frac{f_\ell (n)}{\log\log n} >0.$$
 \end{remark}

By employing a deep result of  Maynard \cite[Theorem 3.1]{Ma2}, we
prove the following result which is used to prove Theorem
\ref{thm1}. This is of interest itself.

\begin{theorem}\label{M}  Let $0<\alpha < \frac 14$. Then there exist a positive constant
$C_0$ depending only on $\alpha $ such that for sufficiently
large number $x$ and $C_0\le k\leqslant (\log x)^{\alpha }$,  if
$\{ u_1n+v_1,\cdots , u_kn+v_k\}$ is a set of $k$ distinct linear
functions with $(u_i, v_i)=1$, $0<u_i<0.9\log\log x$ and
$0\leqslant v_i\leqslant x^{\alpha }$ for all $1\leqslant
i\leqslant k$, then
\begin{equation*}
\#\left\{x\le n<2x:|\{ u_1n+v_1,\dots , u_kn+v_k \}\cap
\mathbb{P}|\geqslant C_0^{-1} \log k \right\}\gg x\exp (-C_0 (\log
x)^\alpha ).
\end{equation*}
\end{theorem}
\begin{remark} The constraint $0<u_i<0.9\log\log x$ in Theorem \ref{M} could be relaxed to $u_i<(\log x)^{\beta}$ and $P(u_i)<0.9\log\log x$ due to Proposition \ref{Pro}, where $\beta$ is any positive number and $P(u_i)$ represents the largest prime factor of $u_i$.
\end{remark}

It should be mentioned that we do not require $\{ u_1n+v_1,\dots ,
u_kn+v_k\}$ to be an admissible set of linear functions in Theorem
\ref{M} (for the definition of the admissible set of linear functions, see the first paragraph of Section 2 below). This will bring convenience in future applications. Below, we give an application of it as an illustration of the convenience. Let $\pi(X;q,a)$ be the number of primes in the arithmetic progression $(qk+a)_{k=0}^{\infty}$ up to $X$ and $\varphi(n)$ be the Euler totient function.

\begin{corollary}\label{app} Let $\varepsilon$ be an arbitrarily small positive number, $a$ a positive
integer and let $x> e^e, y>1$ be numbers satisfying that at least
one of $x$ and $y$ is sufficiently large. Suppose that $1\le a\le
q<\min\left\{y^{1-\varepsilon},0.9\log\log x\right\}$ with
$(a,q)=1$, then there are $\gg x\exp\left(-\sqrt{\log x}\right)$
integers $n\in [x,2x)$ such that
\begin{align}\label{appeq}
\pi(qn+y;q,a)-\pi(qn;q,a)\gg\log y,
\end{align}
where the implied constant depends only on $\varepsilon$.
\end{corollary}
\begin{remark}
Corollary \ref{app} is an expansion of Maynard's result \cite[Theorem 3.2]{Ma2} provided that at least one of $x$ and $y$ is sufficiently large. More precisely,
if we take $q=a=1$ in Corollary \ref{app}, it reduces to the result
\begin{align}\label{reeq1}
\pi(n+y)-\pi(n)\gg \log y
\end{align}
for at least $\gg x\exp\left(-\sqrt{\log x}\right)$ integers $n\in [x,2x)$
obtained by Maynard, provided that at least one of $x$ and $y$ is sufficiently large. \end{remark}

Our paper will be organized as follows. In Section \ref{sec1}, we will give a lemma (Lemma \ref{CD1}) which
illustrates how to pick out the admissible sets from an
arbitrarily given set of linear functions. The lemma itself is elementary but removed the restriction of the admissible condition of distinct linear functions. Then we focus on verifying the conditions of {\bf Hypothesis 1} in the
Maynard theorem. This section would be the key ingredient of our article. In Section \ref{sec2}, we give the proof of Theorem \ref{M} via Proposition \ref{Pro}. In Section \ref{sec4}, we will use Theorem \ref{M} to prove
Theorem \ref{thm1} and give proofs to its corollaries. In Section \ref{111}, we explain that our results can be refined in some sparse sets rather than the whole natural numbers.
In the last section, some heuristic assumptions towards our research object are discussed.

\section{Clusters of primes in subsets}\label{sec1}

 The proof of Theorem \ref{M} follows from a quantitative
result of the admissible sets involving the distributions of
primes established by Maynard \cite{Ma2}. Firstly, we introduce
the notion of the admissible sets of linear functions. Let
 $\mathcal{L}=\{ u_in+v_i : 1\leqslant
i\leqslant k\}$ be a set of linear functions with $u_i\geqslant 1$
$(1\leqslant i\leqslant k)$. We call $\mathcal{L}$ an admissible
set of linear functions, if for each prime $p$ there is some
$m_p\in \mathbb{Z}$ such that
$$p\nmid\prod_{i=1}^{k}(u_im_p+v_i).$$

Below, we give a lemma which illustrates how to pick out the
admissible sets from an arbitrarily given set of linear functions.

\begin{lemma}\label{CD1} If $u_1n+v_1,\dots , u_sn+v_s$ are arbitrarily different $s(\geqslant 5)$
linear functions with $(u_i, v_i)=1$ and $u_i\geqslant 1$
$(1\leqslant i\leqslant s)$, then there are at least
$\gg\frac{s}{\log s}$ of them constituting an admissible set of
linear functions, where the implied constant is absolute.
\end{lemma}
\begin{proof}
 Let $p_i$ be the $i$-th prime.
If $p_i\mid u_j$, then let $\ell_{i,j}=v_j$. If $p_i\nmid u_j$,
then let $\ell_{i,j}=u_{i,j}v_j$, where $u_j u_{i,j}\equiv
1\pmod{p_i}$. For $p_1$, there is a residue class modulo $p_1$
containing at most $\lfloor s/p_1\rfloor$ of
$\ell_{1,1},\ell_{1,2},...,\ell_{1,s}$. So at least $s-\lfloor
s/p_1\rfloor$ of $$\ell_{1,1},\ell_{1,2},...,\ell_{1,s}$$ occupy at
most $p_1-1$ residue classes modulo $p_1$. Denoting by
$s_1=s-\lfloor s/p_1\rfloor$, without loss of generality, we may
assume that
$$\ell_{1,1},\ell_{1,2},...,\ell_{1,s_1}$$
occupy at most $p_1-1$ of the residue classes modulo $p_1$.
Similarly,  at least $s_2=s_1-\lfloor s_1/p_2\rfloor$ elements of
$$\ell_{2,1},\ell_{2,2},...,\ell_{2,s_1}$$  occupy at most $p_2-1$ residue classes modulo
$p_2$. Without loss of generality, we may assume
$$\ell_{2,1},\ell_{2,2},...,\ell_{2,s_2}$$ occupy at most $p_2-1$ of the residue classes modulo $p_2$.
Continuing this process, let $t$ be the least integer with
$s_t<p_{t+1}$, where $s_0=s$. We have
$$s_j=s_{j-1}-\lfloor s_{j-1}/p_j\rfloor \quad (1\leqslant j\leqslant
t).$$
Now we have
\begin{align}\label{equ1}
s_t\geqslant
s_{t-1}\left(1-\frac{1}{p_t}\right)\geqslant\cdots\geqslant
s\left(1-\frac{1}{p_1}\right)\cdots\left(1-\frac{1}{p_t}\right)\gg\frac{s}{\log
p_t}
\end{align}
by the Mertens estimate. By the definition of $t$ we have
$s_{t-1}\geqslant p_{t}$, thus
\begin{align}\label{equ2}
s_t\geqslant s_{t-1}\left(1-\frac{1}{p_t}\right)\geqslant
p_t\left(1-\frac{1}{p_t}\right)= p_t-1.
\end{align}
From equations (\ref{equ1}) and (\ref{equ2}), it can be concluded
that
$$s_t\gg\frac{s}{\log p_t}\geqslant\frac{s}{\log (s_t+1)}\geqslant\frac{s}{\log s}.$$
Finally, we show that
$$\{ u_1n+v_1, \dots , u_{s_t}n+v_{s_t}\}$$
is an admissible set of linear functions.
For any positive integer $i\le t$, since
$$\{ \ell_{i,1},\ell_{i,2},\dots ,\ell_{i,s_t} \} \subseteq \{ \ell_{i,1},\ell_{i,2},\dots ,\ell_{i,s_i}
\} ,$$ it follows that
$$ \ell_{i,1},\ell_{i,2},\dots ,\ell_{i,s_t}  $$
occupy at most $p_i-1$ of the residue classes modulo $p_i$. For
$i>t$, we have $p_i\ge p_{t+1}>s_t$ and $$ \ell_{i,1},\ell_{i,2},\dots
,\ell_{i,s_t}  $$ occupy at most $p_i-1$ of the residue classes
modulo $p_i$. In all cases, there exists an integer $m_i$ such
that
$$m_i+\ell_{i,j}\not\equiv 0\pmod{p_i},\quad 1\le j\le s_t.$$
If $p_i\mid u_j$, then by $(u_j, v_j)=1$, we have $p_i\nmid v_j$.
Thus, $p_i\nmid u_jm_i+v_j$.  If $p_i\nmid u_j$, then
$$u_jm_i+v_j\equiv u_jm_i+u_{i,j} u_j v_j\equiv u_j
(m_i+\ell_{i,j}) \not\equiv 0\pmod{p_i}.$$ In all cases, we have
$$u_jm_i+v_j\not\equiv 0\pmod{p_i},\quad 1\le j\le s_t.$$
That is,
$$p_i\nmid (u_1m_i+v_1)\cdots (u_{s_t}m_i+v_{s_t}).$$
Therefore, $$\{ u_1n+v_1, \dots , u_{s_t}n+v_{s_t}\}$$ is an
admissible set of linear functions.

This completes the proof of Lemma \ref{CD1}.
\end{proof}

In order to prove Theorem \ref{M}, we will employ a deep result of
Maynard \cite{Ma2}.  For convenience of the readers, we use the
same symbols used by Maynard as far as possible (but there are
some minor adjustments of them). Let $\mathcal{A}$ be a subset of
integers and $\mathcal{P}$ a subset of primes. Let $\mathcal{L}=\{
u_in+v_i : 1\leqslant i\leqslant k\}$ and $L_i (n)=u_in+v_i$
$(1\leqslant i\leqslant k)$ and let $\varphi(n)$ be the Euler
totient function. Write
$$\mathcal{A}[x]=\#\{n\in\mathcal{A}:x\leqslant n<2x\},$$
$$\mathcal{A}[x;q,a]=\#\{n\in\mathcal{A}:x\leqslant n<2x,n\equiv a\pmod{q}\},$$
$$L_i(\mathcal{A})=\{L_i(n):n\in \mathcal{A}\},~~\varphi_{L_i}(q)=\varphi(qu_i)/\varphi(u_i),$$
$$\mathcal{P}_{L_i,\mathcal{A}}[x]=\#\{u_in+v_i\in  \mathcal{P}: x\leqslant n<2x,n\in\mathcal{A}\},$$
$$\mathcal{P}_{L_i,\mathcal{A}}[x;q,a]=\#\{ u_in+v_i\in\mathcal{P}: x\leqslant n<2x,n\equiv a\pmod{q},n\in\mathcal{A}\}.$$
Maynard \cite{Ma2} introduced the following {\bf Hypothesis 1} for
well-distributed sets.

\begin{Hypothesis}[$\mathcal{A},\mathcal{P},\mathcal{L},B,x, \theta$]\label{Hy} Let $k=\#\mathcal{L}$.\\
$\mathrm{(1)}~~\mathcal{A}$ is well--distributed in arithmetic progressions: we have
$$\sum_{q\leqslant x^{\theta}}\max_{a}\left|\mathcal{A}[x;q,a]-\frac{\mathcal{A}[x]}{q}\right|\ll \frac{\mathcal{A}[x]}{(\log x)^{100k^2}}.$$
$\mathrm{(2)}$ Primes in $L_i(\mathcal{A})\bigcap \mathcal{P}$ are well distributed in most arithmetic progressions: for any $L_i(n)\in\mathcal{L}$ we have
$$\sum_{\substack{q\leqslant x^{\theta}\\(q,B)=1}}\max_{(L_i(a),q)=1}\left|\mathcal{P}_{L_i,\mathcal{A}}[x;q,a]-\frac{\mathcal{P}_{L_i,\mathcal{A}}[x]}{\varphi_{L_i}(q)}\right|\ll \frac{\mathcal{P}_{L_i,\mathcal{A}}[x]}{(\log x)^{100k^2}}.$$
$\mathrm{(3)}~~\mathcal{A}$ is not too concentrated in any arithmetic progression: for any $q<x^{\theta}$ we have
$$\mathcal{A}[x;q,a]\ll \frac{\mathcal{A}[x]}{q}.$$
\end{Hypothesis}

Based on the above {\bf Hypothesis 1}, Maynard \cite{Ma2}
established the following theorem.

\begin{theorem}[\cite{Ma2}, Theorem 3.1]\label{Maynard} Let $\alpha>0$ and $0<\theta<1$.
Let $x$, $B$  be integers and let $\mathcal{L} =\{ u_1n+v_1,\dots
, u_kn+v_k \}$ be an admissible set of linear functions. Suppose
that $k\leqslant (\log x)^{\alpha}$, $0\leqslant u_i,v_i\leqslant
x^{\alpha}$ for all $1\leqslant i\leqslant k$ and $1\leqslant
B\leqslant x^{\alpha}$. Then there is a constant $C$ depending
only on $\alpha$ and $\theta$ such that the following holds: If
$k\geqslant C$ and
$(\mathcal{A},\mathcal{P},\mathcal{L},B,x,\theta)$ satisfies {\bf
Hypothesis 1}, and if $\delta>(\log k)^{-1}$ is such that
$$\frac{1}{k}\frac{\varphi(B)}{B}\sum_{i=1}^{k}\frac{\varphi(u_i)}{u_i}\mathcal{P}_{L_i,\mathcal{A}}[x]
\geqslant \delta\frac{\mathcal{A}[x]}{\log x},$$
then
$$\#\{n\in \mathcal{A}: x\leqslant n<2x,|\{ L_1(n),L_2(n),...,L_k(n)\}\cap\mathcal{P}|\geqslant C^{-1}\delta \log k\}
\gg \frac{\mathcal{A}[x]}{(\log x)^ke^{Ck}} .$$
\end{theorem}

By employing Lemma \ref{CD1}, we can remove the condition that
$\mathcal{L} =\{ u_1n+v_1,\dots , u_kn+v_k \}$ is an admissible
set of linear functions in Theorem \ref{Maynard}. That is,

\begin{theorem}\label{Maynard1} Let $\alpha>0$ and $0<\theta<1$.
Let $x$, $B$  be integers and let $\mathcal{L} =\{ u_1n+v_1,\dots
, u_kn+v_k \}$ be a set of linear functions with $(u_i, v_i)=1$
for all $i$. Suppose that $k\leqslant (\log x)^{\alpha}$,
$0\leqslant u_i,v_i\leqslant x^{\alpha}$ for all $1\leqslant
i\leqslant k$ and $1\leqslant B\leqslant x^{\alpha}$. Then there
is a positive constants $C_1$ depending only on
$\alpha$ and $\theta$ such that the following holds: If
$k\geqslant C_1$ and
$(\mathcal{A},\mathcal{P},\mathcal{L},B,x,\theta)$ satisfies {\bf
Hypothesis 1}, and if $\delta>2(\log k)^{-1}$ is such that
$$\frac{\varphi(B)\varphi(u_i)}{Bu_i}\mathcal{P}_{L_i,\mathcal{A}}[x]
\geqslant \delta\frac{\mathcal{A}[x]}{\log x}$$
for any $1\leqslant i\leqslant k$, then
$$\#\{n\in \mathcal{A}: x\leqslant n<2x,|\{ L_1(n),L_2(n),...,L_k(n)\}\cap\mathcal{P}|\geqslant C_1^{-1}\delta\log k\}
\gg \frac{\mathcal{A}[x]}{e^{C_1 (\log x)^\alpha }} .$$
Actually, we have the following more explicit bound
$$\#\{n\in \mathcal{A}: x\leqslant n<2x,|\{ L_1(n),L_2(n),...,L_k(n)\}\cap\mathcal{P}|\geqslant C_1^{-1}\delta\log k\}
\gg \frac{\mathcal{A}[x]}{(\log x)^ke^{C_1k}} .$$
\end{theorem}

\begin{proof}
Let $C$ be as in Theorem \ref{Maynard}.
By Lemma \ref{CD1}, there are at least $C'{k}/{\log k}$ of
$u_1n+v_1,\dots , u_kn+v_k$ constituting an admissible set of
linear functions, where $C'$ is an absolute positive constant. Let
$$\ell =\left\lceil C'\frac{k}{\log k} \right\rceil .$$
We take an admissible set of $\ell $ linear functions
$$\{u_{i_1}n+v_{i_1},\dots, u_{i_\ell}n+v_{i_\ell}\}$$
from
$u_1n+v_1,\dots , u_kn+v_k$.
Since $$\frac{\varphi(B)\varphi(u_i)}{Bu_i}\mathcal{P}_{L_i,\mathcal{A}}[x]
\geqslant \delta\frac{\mathcal{A}[x]}{\log x}$$
for any  $1\leqslant i\leqslant k$, it follows that
$$\frac{1}{\ell}\frac{\varphi(B)}{B}\sum_{j=1}^{\ell}\frac{\varphi(u_{i_{j}})}{u_{i_j}}\mathcal{P}_{L_i,\mathcal{A}}[x]
\geqslant \delta\frac{\mathcal{A}[x]}{\log x}.$$
Recall that $k\le (\log x)^\alpha $, we
have
$$ \frac{k}{\log k}\le \alpha^{-1} \frac{(\log x)^\alpha}{\log\log
x} $$ for $k\ge 3$. We take a constant $C_1\geqslant\max\{5,C\}$ such that for all
$k\ge C_1$, we have $\ell \ge C$,
$$\log \ell \ge \log \left( C' \frac{k}{\log k} \right) >\frac 12 \log k,$$
$$C^{-1}\delta\log \ell\ge C^{-1}\delta \frac{1}{2}\log k\ge C_1^{-1}\delta\log k$$
and
\begin{eqnarray*}(\log x)^\ell e^{C\ell }=\exp (\ell \log \log x +C\ell )&\le & \exp
\left( 2C'\frac{k}{\log k} \log\log x+2CC'\frac{k}{\log k}
\right)\\
&\le & \exp (C_1(\log x)^\alpha ).\end{eqnarray*} By $\ell \le k$, we have
$$(\log x)^\ell e^{C\ell }\le (\log x)^k e^{Ck}\le (\log x)^k e^{C_1k}.$$
Theorem \ref{Maynard1} now follows from Theorem \ref{Maynard} by noting
$$\delta>2(\log k)^{-1}>(\log \ell)^{-1}.$$
\end{proof}

As mentioned by Maynard \cite{Ma2}, Theorem \ref{Maynard} (hence Theorem \ref{Maynard1}) can apply to vary sparse sets $\mathcal{A}$ and for such sets the major obstacle is in establishing  {\bf Hypothesis 1}. Maynard \cite{Ma2} (in his proof of Theorem 3.2 therein) illustrated that $\mathcal{A}=\mathbb{Z}$ with $u_i=1~(1\leqslant i\leqslant k)$ is applicable. In next proposition, we shall follow the idea of Maynard and give an extension of his argument by choosing $u_i$ from a larger range.

For any integer $n>1$, let $P(n)$ be the largest prime factor of
$n$.

\begin{proposition}\label{Pro} Suppose that $0<\alpha<1/4$, $\beta >0$ and $\mathcal{L}=\{ u_1n+v_1, \dots , u_kn+v_k\}$
is a set of linear functions with $k\le (\log x)^\alpha$ and
$(u_i, v_i)=1$ such that
$$1\leqslant u_i\leqslant (\log x)^{\beta },\quad P(u_i)<0.9
\log\log x, \quad 0\leqslant v_i\leqslant x^{\alpha }$$ for all
$1\leqslant i\leqslant k$. Then for any sufficiently large $x$,
there is an integer $B$ satisfying
$$\frac{\varphi(B)}{B}=1+O\left(\frac{1}{\log\log
x}\right)$$ such that {\bf Hypothesis 1} is satisfied for
$[\mathcal{A},\mathcal{P},\mathcal{L},B,x,\theta]$ with
$$\mathcal{A}=\mathbb{Z}, \quad \mathcal{P}=\mathbb{P}, \quad \theta=\frac{1}{3}.$$
\end{proposition}

\begin{proof}
I. {\it Verification of condition} (3) in {\bf Hypothesis 1.}\\
It is trivial since for $q<x^\theta$,
$$\mathcal{A}[x;q,a]=\frac{x}{q}+O(1)\ll\frac{\mathcal{A}[x]}{q}.$$
II. {\it Verification of condition} (1) in {\bf Hypothesis 1.}\\
The first one is also easy to verify because
$$\sum_{q\leqslant x^{\theta}}\max_{a}\left|\mathcal{A}[x;q,a]-\frac{\mathcal{A}[x]}{q}\right|\ll \sum_{q\leqslant x^{1/3}}1\ll \frac{\mathcal{A}[x]}{(\log x)^{100k^2}}.$$
III. {\it Verification of condition} (2) in {\bf Hypothesis 1.}\\
For a large $x$, there is at most one real primitive
character $\widetilde{\chi}$ to a modulus $q_0\leqslant
e^{c_1\sqrt{\log x}}$ such that the corresponding $L$--function
$L(s,\widetilde{\chi})$ has a real zero larger than $1-c_2(\log
x)^{-1/2}$ via the Landau--Page theorem (see for example
\cite[Page 95]{Da}), where $c_1$ and $c_2$ are two  suitable
absolutely positive constants which will be given later. With this
fact, we take
\begin{equation*}
B=
\begin{cases}
P(q_0),~~~&\text{if~such~an~exceptional~zero~does~exist,}\\
1,~~~&\text{otherwise.}
\end{cases}
\end{equation*}
This exceptional zero if exists, its modulus $q_0$ should satisfy
$q_0(\log q_0)^4\gg\log x$ (see for example \cite[Page 124]{Da}),
which means that $q_0\gg \log x/(\log\log x)^4$. Moreover, if such
an exceptional zero does exist, its modulus $q_0$ is square--free apart
from a possible factor of at most $4$, otherwise
$\widetilde{\chi}$ would not be primitive (see for example
\cite[Page 40]{Da}). In view of these facts, on one aspect we have
$$B\leqslant q_0\leqslant e^{c_1\sqrt{\log x}}.$$
On the other aspect, let $q_0'=q_0/(q_0,4)$, then $q_0'$ is a
square--free number satisfying $B=P(q_0')$ and $q_0'\gg \log
x/(\log\log x)^4$. From which we know that $B>0.9 \log\log x$,
otherwise we shall have
$$q_0'\leqslant \prod_{p\leqslant 0.9\log\log x}p=\exp \left(\sum_{p\leqslant 0.9\log\log x}\log p\right) <(\log x)^{0.99},$$
which is certainly a contradiction with $q_0'\gg \log x/(\log\log
x)^4$ for sufficiently large $x$. Thus we can conclude that
$0.9\log\log x<B\leqslant e^{c_1\sqrt{\log x}}.$ It follows that
$$\frac{\varphi(B)}{B}=1+O\left(\frac{1}{\log\log
x}\right).$$ Let $\pi(X)$ be the the number of primes below $X$
and
$$\pi(X;m,\ell)=\sum_{\substack{p\leqslant X,~~p\in\mathbb{P} \\ p\equiv
\ell\pmod{m}}}1.$$ Suppose that $(L_i(a),q)=1$. Then
$(L_i(a),qu_i)=1$. We have
\begin{align*}
\mathcal{P}_{L_i,\mathcal{A}}[x;q,a]&=\sum_{\substack{x\leqslant n<2x\\ n\equiv a\pmod{q}\\ u_in+v_i\in \mathbb{P}}}1
=\sum_{\substack{u_ix+v_i\leqslant m< 2u_ix+v_i\\m\in \mathbb{P}\\m\equiv L_i(a)\pmod{qu_i}}}1\\
&=\pi(2u_ix+v_i; qu_i,L_i(a))-\pi(u_ix+v_i;qu_i,L_i(a))+r_{x,i}
\end{align*}
 and
\begin{align*}
\mathcal{P}_{L_i,\mathcal{A}}[x]&=\sum_{\substack{x\leqslant n<2x\\ u_in+v_i\in \mathbb{P}}}1=\sum_{\substack{u_ix+v_i\leqslant m< 2u_ix+v_i\\m\in \mathbb{P}\\m\equiv v_i\pmod{u_i}}}1\\
&=\pi(2u_ix+v_i;u_i,v_i)-\pi(u_ix+v_i;u_i,v_i)+r'_{x,i},
\end{align*}
where $|r_{x,i}|\leqslant 1$ and $|r'_{x,i}|\leqslant 1$. It
follows from $\varphi_{L_i}(q)\ge 1$ that
\begin{eqnarray*}&&\left|\mathcal{P}_{L_i,\mathcal{A}}[x;q,a]
-\frac{\mathcal{P}_{L_i,\mathcal{A}}[x]}{\varphi_{L_i}(q)}\right|\\
&\le & \left| \pi(2u_ix+v_i;
qu_i,L_i(a))-\frac{\pi(2u_ix+v_i;u_i,v_i)}{\varphi_{L_i}(q)}\right|\\
&& +\left| \pi(u_ix+v_i;
qu_i,L_i(a))-\frac{\pi(u_ix+v_i;u_i,v_i)}{\varphi_{L_i}(q)}\right|
+2.
\end{eqnarray*}
Let $\psi(x)=\sum_{n\leqslant x}\Lambda(n)$ and
$$\psi(x;m,\ell)=\sum_{\substack{n\leqslant x\\n\equiv
\ell\pmod{m}}}\Lambda(n),$$
$$\vartheta(x;m,\ell)=\sum_{\substack{p\leqslant x\\p\equiv
\ell\pmod{m}}}\log p.$$
Here as usual, $\Lambda(n)$ is the
Mangoldt function.  Hence

\begin{align}\label{eq0422-1}&\sum_{\substack{q\leqslant
x^{\theta}\\(q,B)=1}}\max_{(L_i(a),q)=1}\left|\mathcal{P}_{L_i,\mathcal{A}}[x;q,a]
-\frac{\mathcal{P}_{L_i,\mathcal{A}}[x]}{\varphi_{L_i}(q)}\right|\nonumber\\
=&\sum_{\substack{q\leqslant
x^{1/3}\\(q,B)=1}}\max_{(L_i(a),qu_i)=1}\left| \pi(2u_ix+v_i;
qu_i,L_i(a))-\frac{\pi(2u_ix+v_i;u_i,v_i)\varphi(u_i)}{\varphi(qu_i)}\right|\nonumber\\
&~ +\sum_{\substack{q\leqslant
x^{1/3}\\(q,B)=1}}\max_{(L_i(a),qu_i)=1}\left| \pi(u_ix+v_i;
qu_i,L_i(a))-\frac{\pi(u_ix+v_i;u_i,v_i)\varphi(u_i)}{\varphi(qu_i)}\right|
+O(x^{1/3}),\nonumber\\
\ll&\sum_{\substack{q\leqslant
x^{1/3}\\(q,B)=1}}\max_{(L_i(a),qu_i)=1}\max_{\sqrt{x}<y\leqslant 2u_ix+v_i}\left| \psi(y;
qu_i,L_i(a))-\frac{\psi(y;u_i,v_i)\varphi(u_i)}{\varphi(qu_i)}\right|\nonumber\\
&~ +\sum_{\substack{q\leqslant
x^{1/3}\\(q,B)=1}}\max_{(L_i(a),qu_i)=1}\max_{\sqrt{x}<y\leqslant u_ix+v_i}\left| \psi(y;
qu_i,L_i(a))-\frac{\psi(y;u_i,v_i)\varphi(u_i)}{\varphi(qu_i)}\right|
+x^{6/7},\nonumber\\
=&S_1(x)+S_2(x)+O(x^{6/7}),~~\text{say,}
\end{align}
where the last but one step comes from integration by parts. In fact, let  $z$ be a number greater than $2$ and $(m,t)=1$. Then
\begin{align}\label{eq0423-1}\pi(z;m,t)&=\sum_{\substack{p\leqslant z\\p\equiv t \pmod{m}}}1=\frac{\vartheta(z;m,t)}{\log z}+\int_{2}^{z}\frac{\vartheta(u;m,t)}{u\log ^2u}du\nonumber\\
&=\frac{\psi(z;m,t)}{\log z}+\int_{2}^{z}\frac{\psi(u;m,t)}{u\log ^2u}du+O(\sqrt{z}).
\end{align}
From equation (\ref{eq0423-1}), we have
\begin{align*}
&\left| \pi(u_ix+v_i;
qu_i,L_i(a))-\frac{\pi(u_ix+v_i;u_i,v_i)\varphi(u_i)}{\varphi(qu_i)}\right|\\
\leqslant&\frac{1}{\log (u_ix+v_i)}\left| \psi(u_ix+v_i;
qu_i,L_i(a))-\frac{\psi(u_ix+v_i;u_i,v_i)\varphi(u_i)}{\varphi(qu_i)}\right|\\
&~+\int_{2}^{u_ix+v_i}\left| \frac{\psi(u;
qu_i,L_i(a))-\psi(u;u_i,v_i)\varphi(u_i)/\varphi(qu_i)}{u\log^2u}\right|du+O(\sqrt{u_ix})\\
\ll&\max_{\sqrt{x}<y\leqslant u_ix+v_i}\left| \psi(y;
qu_i,L_i(a))-\frac{\psi(y;u_i,v_i)\varphi(u_i)}{\varphi(qu_i)}\right|\int_{\sqrt{x}}^{u_ix+v_i}\frac{1}{u\log^2u}du+\sqrt{u_ix}\nonumber\\
\ll&\max_{\sqrt{x}<y\leqslant u_ix+v_i}\left| \psi(y;
qu_i,L_i(a))-\frac{\psi(y;u_i,v_i)\varphi(u_i)}{\varphi(qu_i)}\right|+\sqrt{u_ix}.
\end{align*}
The same estimates can be used when $u_ix+v_i$ is replaced by $2u_ix+v_i$.
By the triangle inequality, we have

\begin{align*}S_1(x)=&\sum_{\substack{q\leqslant
x^{1/3}\\(q,B)=1}}\max_{(L_i(a),qu_i)=1}\max_{\sqrt{x}<y\leqslant 2u_ix+v_i}\left| \psi(y;
qu_i,L_i(a))-\frac{\psi(y)}{\varphi(qu_i)}\right|\\
&~+\sum_{\substack{q\leqslant
x^{1/3}\\(q,B)=1}}\max_{(L_i(a),qu_i)=1}\max_{\sqrt{x}<y\leqslant 2u_ix+v_i}\left|\frac{\psi(y;u_i,v_i)\varphi(u_i)}{\varphi(qu_i)}-\frac{\psi(y)}{\varphi(qu_i)}\right|\\
=&S_{11}(x)+S_{12}(x),~\text{say}
\end{align*}
and
\begin{align*}S_2(x)=&\sum_{\substack{q\leqslant
x^{1/3}\\(q,B)=1}}\max_{(L_i(a),qu_i)=1}\max_{\sqrt{x}<y\leqslant u_ix+v_i}\left| \psi(y;
qu_i,L_i(a))-\frac{\psi(y)}{\varphi(qu_i)}\right|\\
&~+\sum_{\substack{q\leqslant
x^{1/3}\\(q,B)=1}}\max_{(L_i(a),qu_i)=1}\max_{\sqrt{x}<y\leqslant u_ix+v_i}\left|\frac{\psi(y;u_i,v_i)\varphi(u_i)}{\varphi(qu_i)}-\frac{\psi(y)}{\varphi(qu_i)}
\right|\\
=&S_{21}(x)+S_{22}(x),~\text{say.}
\end{align*}
Recall that $1\leqslant u_i\leqslant (\log x)^{\beta }$ and
$1\leqslant v_i\leqslant x^{\alpha }$, by the Siegel--Walfisz
theorem (see for example \cite[Page 133]{Da}) we have

$$\psi(y;u_i,v_i)=\frac{\psi(y)}{\varphi(u_i)}+O(ye^{-c_3\sqrt{\log y}})$$
for any $y>\sqrt{x}$, where $c_3$ is an absolute positive constant. From which we deduce
that
\begin{eqnarray}\label{eq0422-2}
S_{12}(x)+S_{22}(x)&\ll & \sum_{\substack{q\leqslant
x^{1/3}\\(q,B)=1}}\max_{(L_i(a),qu_i)=1}\max_{\sqrt{x}<y\leqslant 2u_ix+v_i}\frac{\varphi(u_i)}{\varphi(qu_i)}ye^{-c_3\sqrt{\log
y}}\nonumber\\
&\ll & \sum_{q\leqslant
x^{1/3}}\frac{1}{\varphi(q)}xe^{-c_4\sqrt{\log
x}}\nonumber\\
&\ll & xe^{-c_5\sqrt{\log x}},
\end{eqnarray}
where $c_5< c_4$ are two positive constants. We now estimate the summation
$S_{21}(x)$. Denoting by
$$E_1^*(z,q)=\max_{(L_i(a),q)=1}\max_{\sqrt{x}<y\leqslant z}\left|\psi(y;q,L_i(a))-\frac{\psi(y)}{\varphi(q)}\right|$$
and
$$E_2^*(z,q)=\max_{(L_i(a),q)=1}\max_{y\leqslant z}\left|\psi(y;q,L_i(a))-\frac{\psi(y)}{\varphi(q)}\right|,$$
then
\begin{align}\label{eq0422-03}
S_{21}(x)&=\sum_{\substack{q\leqslant
x^{1/3}\\(q,B)=1}}E_1^*\left(u_ix+v_i,qu_i\right)\nonumber\\&\leqslant\sum_{\substack{q\leqslant
x^{1/3}(\log x)^{\beta}\\(q,B)=1}}E_1^*\left(u_ix+v_i,q\right)\nonumber\\&\leqslant\sum_{\substack{q\leqslant
x^{1/3}(\log x)^{\beta}\\(q,B)=1}}E_2^*\left(u_ix+v_i,q\right),
\end{align}
where the last but one estimate comes from the fact that $(B,u_i)=1$ for $
P(u_i)<0.9 \log\log x <B$ since $B$ is a prime larger than
$0.9\log\log x$. Using again the triangle inequality, we find that
\begin{align}\label{formula1}
E_2^*(z,q)\le\max_{(L_i(a),q)=1}\max_{y\leqslant z}\left|\psi(y;q,L_i(a))-\frac{y}{\varphi(q)}\right|+\max_{(L_i(a),q)=1}\max_{y\leqslant z}\left|\frac{y}{\varphi(q)}-\frac{\psi(y)}{\varphi(q)}\right|.
\end{align}
The elaborate prime number theorem (see for example \cite[Page 112]{Da}) shows that
\begin{align}\label{formula2}
\max_{(L_i(a),q)=1}\max_{y\leqslant z}\left|\frac{y}{\varphi(q)}-\frac{\psi(y)}{\varphi(q)}\right|\ll\frac{1}{\varphi(q)}z\exp(-c^*(\log z)^{1/2})
\end{align}
for some constant $c^*>0$. Collecting equations (\ref{eq0422-03}), (\ref{formula1}) and (\ref{formula2}), we have
\begin{align}\label{formula3}
S_{21}(x)\leqslant\sum_{\substack{q\leqslant
x^{1/3}(\log x)^{\beta}\\(q,B)=1}}E^*\left(u_ix+v_i,q\right) +O\left(x\exp\left(-c^{**}\sqrt{\log x}\right)\right),
\end{align}
where $c^{**}$ is a positive constant slightly smaller than $c^*$ and
$$E^*\left(z,q\right)=\max_{(L_i(a),q)=1}\max_{y\leqslant z}\left|\psi(y;q,L_i(a))-\frac{y}{\varphi(q)}\right|.$$
Now, let $\chi$ be the characters with modulus $q$. From \cite[Page 163]{Da}, we know that
\begin{align}\label{formula4}
\sum_{\substack{q\leqslant
Q\\(q,B)=1}}E^*\left(u_ix+v_i,q\right)\ll Q(\log Qx)^{2}+\log x\sum_{\substack{q\leqslant
Q\\(q,B)=1}}\frac{1}{\varphi(q)}\sum_{\chi~\text{primitive}}\max_{y\leqslant u_ix+v_i}|\psi'(y,\chi)|,
\end{align}
where
\begin{equation}\label{Eq.12}
\psi'(y,\chi)=
\begin{cases}
\psi(y,\chi),~~~&\text{if~} \chi\neq\chi_0,\\
\psi(y,\chi_0)-y,~~~&\text{if~} \chi=\chi_0.
\end{cases}
\end{equation}
Here, of course, $\chi_0$ is the trivial character modulo $q$ and $$\psi(x,\chi)=\sum_{n\leqslant x}\Lambda(n)\chi(n).$$
It can be proved via \cite[Page 164]{Da} with $Q=x^{1/3}(\log x)^{\beta}$  that
\begin{align*}
\sum_{\substack{Q_1<q\leqslant
Q\\(q,B)=1}}\frac{1}{\varphi(q)}\sum_{\chi~\text{primitive}}\max_{y\leqslant u_ix+v_i}|\psi'(y,\chi)|\ll (\log x)^4\left(\frac{x(\log x)^\beta}{Q_1}+x^{5/6}(\log
x)^{1+5\beta/6}\right)
\end{align*}
for any number $Q_1\leqslant x^{1/3}(\log x)^{\beta}$. Thus,
the right hand--side of equation (\ref{formula4}) is
\begin{equation}\label{Eq.11}
\ll (\log x)^5\left(\frac{x(\log x)^\beta}{Q_1}+x^{5/6}(\log
x)^{1+5\beta/6}\right)+\log x\sum_{\substack{q\leqslant
Q_1\\(q,B)=1}}\frac{1}{\varphi(q)}\sum_{\chi~\text{primitive}}|\psi'(u_ix+v_i,\chi)|
\end{equation}
for any number $Q_1\leqslant x^{1/3}(\log x)^{\beta}$.
It is trivial
that
\begin{align}\label{Eq.13}
|\psi(u_ix+v_i,\chi_0)-(u_ix+v_i)|&\leqslant |\psi(u_ix+v_i,\chi_0)-\psi(u_ix+v_i)|+|\psi(u_ix+v_i)-(u_ix+v_i)|\nonumber\\
&\ll \sum_{\substack{n\leqslant u_ix+v_i\\(n,q)>1}}\Lambda(n)+(u_ix+v_i)e^{-c_5\sqrt{\log (u_ix+v_i)}}\nonumber\\
&\leqslant (\log q)\log (u_ix+v_i)+(u_ix+v_i)e^{-c_5\sqrt{\log (u_ix+v_i)}}\nonumber\\
&\ll (u_ix+v_i)e^{-c_6\sqrt{\log (u_ix+v_i)}}\nonumber\\
&\ll x e^{-c_7 \sqrt{\log x}} .
\end{align}
from the prime number theorem, where $c_6$ and $c_7$ are two
absolute positive constants. Now, we take $Q_1=e^{c_1\sqrt{\log
x}}$ and so the constraint $(q,B)=1$
 will guarantee that there is no Landau--Siegel zero in the domain $q\leqslant Q_1$. Therefore, for $\chi\neq\chi_0$ we shall have
\begin{equation}\label{Eq.14}
|\psi(u_ix+v_i,\chi)|\ll xe^{-c_8\sqrt{\log x}}
\end{equation}
within the scope $q\leqslant Q_1$ and $(q,B)=1$ (see for example
\cite[Page 132]{Da}), where $c_8$ is an absolute positive
constant. Combing equations (\ref{formula3})--(\ref{Eq.14}), we obtain that
\begin{eqnarray*}S_{21}(x) &\ll & \frac{x(\log x)^{\beta
+5}}{e^{c_1\sqrt{\log x}}}+\log x\sum_{q\leqslant e^{c_1\sqrt{\log
x}}}xe^{-c_9\sqrt{\log
x}}+xe^{-c^{**}\sqrt{\log x}}\\
&\ll & \frac{x(\log x)^{\beta +5}}{e^{c_1\sqrt{\log x}}}+x(\log x)
e^{-(c_9-c_1)\sqrt{\log x}}+xe^{-c^{**}\sqrt{\log x}},
\end{eqnarray*}
where
$c_9=\min\{c_7,c_8\}$. Similarly, we have
$$S_{11}(x)\ll \frac{x(\log x)^{\beta +5}}{e^{c_1\sqrt{\log x}}}+x(\log x)
e^{-(c_{10}-c_1)\sqrt{\log x}}+xe^{-c^{***}\sqrt{\log x}},$$ where $c_{10}$ and $c^{***}$ are two absolute positive
constants. We take $c_1=\frac 12 \min\{ c_9, c_{10}\}$ and
$c_{11}=\min\{\frac 12c_1,c^{**},c^{***}\}$. Then
$$S_{11}(x)\ll xe^{-c_{11}\sqrt{\log x}},\quad S_{21}(x)\ll xe^{-c_{11}\sqrt{\log x}},$$
which is surely more to expectation than the bound of the
condition (2) in {\bf Hypothesis 1}. This can be seen by $k\le
(\log x)^\alpha $,
$$\mathcal{P}_{L_i,\mathcal{A}}[x]=\pi(2u_ix+v_i;u_i,v_i)-\pi(u_ix+v_i;u_i,v_i)+r'_{x,i}\gg\frac{x}{\log x}$$
and
\begin{align*}
(\log
x)^{100k^2-1}&=\exp ((100k^2-1)\log\log x)\\
&\leqslant \exp ( 100
(\log x)^{2\alpha }(\log\log x))\\
& < \exp (c_{11} \sqrt{\log x}).
\end{align*}
This completes the proof of Proposition \ref{Pro}.
\end{proof}

\section{Proof of Theorem \ref{M} and Corollary \ref{app}}\label{sec2}

We now prove Theorem \ref{M} from  Theorem \ref{Maynard1} and Proposition \ref{Pro}.

\begin{proof}[Proof of Theorem \ref{M}] Let
$$\theta =\frac 13, \quad \mathcal{A}=\mathbb{Z},\quad \mathcal{P}=\mathbb{P},\quad
\mathcal{L} =\{ u_1n+v_1,\dots , u_kn+v_k\} .$$ Let $C_1$ be the same one in Theorem \ref{Maynard1}. We take
$\delta =\frac {1}{2}$ and $$C_0= \max\left\{\left\lfloor e^{4}\right\rfloor+1,2C_1\right\},$$ then we shall have $k\geqslant C_1$ and $\delta>2(\log k)^{-1}$ for $k\geqslant C_0$.

By Proposition \ref{Pro}, {\bf Hypothesis 1} is satisfied for
[$\mathcal{A},\mathcal{P},\mathcal{L},B,x, \theta$] chosen above, where $B$ is
the positive integer given in Proposition \ref{Pro}. Recall that for sufficiently large $x$, we have
\begin{align*}\mathcal{P}_{L_i,\mathcal{A}}[x]&=\sum_{\substack{x\leqslant n<2x\\ u_in+v_i\in \mathbb{P}}}1=\sum_{\substack{u_ix+v_i\leqslant m< 2u_ix+v_i\\m\in \mathbb{P}\\m\equiv v_i\pmod{u_i}}}1\\
&\geqslant\pi(2u_ix+v_i;u_i,v_i)-\pi(u_ix+v_i;u_i,v_i)-2\\
&=(1+o(1))\left\{\frac{1}{\varphi(u_i)}\frac{2u_ix+v_i}{\log (2u_ix+v_i)}-\frac{1}{\varphi(u_i)}\frac{u_ix+v_i}{\log (u_ix+v_i)}\right\}\\
&>\frac{2u_i}{3\varphi(u_i)}\frac{x}{\log x}
\end{align*}
by the Siegel--Walfisz theorem (see for example \cite[Page 133]{Da}).
From which we deduce
\begin{align*}\label{eq4}
\frac{\varphi(B)}{B}\frac{\varphi(u_i)}{u_i}\mathcal{P}_{L_i,\mathcal{A}}(x)
\geqslant& \left(1+O\left(\frac{1}{\log\log
x}\right)\right)\frac{\varphi(u_i)}{u_i}\frac{2u_i}{3\varphi(u_i)}\frac{x}{\log x}\\
\geqslant &\frac{2}{3}\left(1+O\left(\frac{1}{\log\log
x}\right)\right)\frac{x}{\log x}\\
>& \frac 1{2} \frac{\mathcal{A}[x]}{\log x}
= \delta\ \frac{\mathcal{A}[x]}{\log x}
\end{align*}
for any $1\leqslant i\leqslant k$, provided that $x$ is sufficiently large. By Theorem
\ref{Maynard1},
\begin{align*}
\#\{ n: x\leqslant n<2x, |\{ u_1n+v_1,\dots ,
u_kn+v_k\}\cap\mathcal{P}|\geqslant C_1^{-1}\delta  \log k \} &\gg x
\exp \left( -C_1 (\log x)^{\alpha } \right)\\
&\ge x\exp \left( -C_0 (\log x)^{\alpha } \right).
\end{align*}
The proof of Theorem \ref{M} is complete by the fact $C_1^{-1}\delta=\frac{1}{2}C_1^{-1}\ge C_0^{-1}$.
\end{proof}

As an application of Theorem \ref{M}, we prove Corollary \ref{app}.

\begin{proof}[Proof of Corollary \ref{app}]
Let us first fix a sufficiently large number $m_0$ and prove the
corollary via Theorem \ref{M} for the case that $x$ is larger than
$m_0$. Without loss of generality, we may assume that
$0<\varepsilon <1/100$.

We divide into the following fours cases:

{\bf Case I.} $x\ge m_0$ and $(C_0+1)^{2/\varepsilon}\le y\le
(\log x)^{3}$. We take $\alpha=\frac15$, $k=\left\lfloor
y^{\varepsilon/2}\right\rfloor$ and $$u_i=q, \quad v_i=(i-1)q+a
\quad (1\le i\le k)$$ in Theorem \ref{M}. Then we have $$C_0\le
k\le y^{\varepsilon/2}\le (\log x)^{3\varepsilon/2}<(\log
x)^{1/5}$$ and $$v_i\le kq\le (\log x)^{3\varepsilon/2} \cdot
0.9\log\log x <(\log x)^{1/5}.$$ Thus
\begin{align*} \#&\{x\le
n<2x:\left|\{qn+a,\dots,qn+(k-1)q+a\}\cap\mathbb{P}\right|>C_0^{-1}\log
k\}\\&\gg x\exp\left(-C_0(\log
x)^{1/5}\right)>x\exp\left(-\sqrt{\log x}\right)
\end{align*}
by Theorem \ref{M}. It follows that there are at least $x\exp\left(-\sqrt{\log x}\right)$ integers $n\in [x,2x)$ such that
\begin{equation}\label{6-5}
\pi(qn+y;q,a)-\pi(qn;q,a)>C_0^{-1}\log k\gg_{\varepsilon} \log y.
\end{equation}

{\bf Case II.} $x\ge m_0$ and $y< (C_0+1)^{2/\varepsilon}$. By the
Siegel--Walfisz theorem, there are at least
$\frac{1}{2}\frac{qx}{\varphi(q)\log x}$ integers $n\in [x,2x)$
such that $qn+a$ is a prime. In other words, for at least
$\frac{1}{2}\frac{qx}{\varphi(q)\log x}(>\frac{1}{2}\frac{x}{\log x})$ integers $n$ we have
$$\pi(qn+y;q,a)-\pi(qn;q,a)\ge 1\gg_{\varepsilon} \log y$$
since $a<q\le y^{1-\varepsilon}<y<  (C_0+1)^{2/\varepsilon}$.

{\bf Case III.} $x\ge m_0$ and $\log^3x<y\le qx/2$. Since
\begin{align*}
\sum_{x\le n<2x}(\pi(qn+y;q,a)-\pi(qn;q,a))&=\sum_{x\le n<2x}\sum_{\substack{qn<p\le qn+y\\p\equiv a\pmod{q}}}1\\
&=\sum_{\substack{qx<p\le 2qx+y\\p\equiv a\pmod{q}}}\sum_{\substack{x\le n<2x\\\frac{p-y}{q}\le n<\frac{p}{q}}}1\\
&>\sum_{\substack{\frac{3}{2}qx<p\le 2qx\\p\equiv a\pmod{q}}}\sum_{\substack{\frac{p-y}{q}\le n<\frac{p}{q}}}1\\
&>\frac{y}{2q}\sum_{\substack{\frac{3}{2}qx<p\le 2qx\\p\equiv a\pmod{q}}}1,
\end{align*}
we know by the Siegel--Walfisz theorem and $q<0.9\log\log x$ that
$$\sum_{x\le n<2x}(\pi(qn+y;q,a)-\pi(qn;q,a))>\frac{xy}{6\varphi(q)\log x}>\frac{xy}{6\log\log x \log x},$$
from which we deduce that at least $x\exp\left(-\sqrt{\log
x}\right)$ integers $n\in [x,2x)$ satisfy our requirement.

{\bf Case IV.} $x\ge m_0$ and $y>qx/2$. It follows from the
Siegel--Walfisz theorem and $q<0.9\log\log x$ that every integer
$n\in [x,2x)$ would satisfy our requirement.

We are left over to treat $x< m_0$. In this case, $y$ is
sufficiently large by the assumption of our corollary, and the
corollary follows trivially from the prime number theorem in
arithmetic progressions.
\end{proof}

\begin{remark} From the proof of Corollary \ref{app} we know that it is nontrivial only for $x$ sufficiently large and $y$ equally to a relatively small power of $\log x$, otherwise the corollary shall be an exercise of the prime number theorem and the Siegel--Walfisz theorem.
\end{remark}

\section{Proofs of Theorem \ref{thm1} and its corollaries}\label{sec4}

\begin{proof}[Proof of Theorem \ref{thm1}]
Let $x$ be a sufficiently large number and $X=(x\log\log x)^{1/\alpha }$. Then $x=\frac{X^\alpha}{\log\log x} $, $a_1<\cdots < a_t\le
\frac{X^\alpha}{\log\log x}$, $$1\le \ell_i<0.9\log\log x<0.9\log\log X$$ and $$t>\left(\log\frac{X^\alpha}{\log\log x}\right)^\alpha >0.9\alpha^\alpha (\log X)^\alpha. $$
Let $k$ be the integer with
\begin{equation}\label{e1}k\le 0.9\alpha^\alpha (\log X)^\alpha
<k+1.\end{equation} Then $k< (\log X)^\alpha$ and $k<t$. Let
 $v_i=\ell_ia_t-a_i$ $(1\le i\le k)$. Then $1\le v_i\le
\ell_ia_t\le X^\alpha$. By Theorem \ref{M}, there is a positive constant $C_0$  depending only on $\alpha$ such that
\begin{equation*}
\#\left\{ X\leqslant n<2X\ \!\!: \!\!\!\ |\{ \ell_1n+v_1,\dots , \ell_kn+v_k \}\cap
\mathbb{P}|\geqslant C_0^{-1}\log k \right\}\gg X\exp (-C_0(\log
X)^\alpha ).
\end{equation*}
That is,
\begin{align*}
\#\left\{ X\leqslant n<2X\!\! : \!\!\ |\{ \ell_1(n+a_t)-a_1,\dots , \ell_k(n+a_t)-a_k
\}\cap \mathbb{P}|\geqslant C_0^{-1}\log k \right\}\\\gg X\exp (-C_0
(\log X)^\alpha ).
\end{align*}
It follows that there are at least $X\exp (-C_0(\log X)^\alpha )$
of integers $n\in [X, 2X)$
such that at least $C_0^{-1}\log k$ of the representations $$\ell_i(n+a_t)=p+a_i \quad (p\in \mathbb{P},1\leqslant i\leqslant t)$$
are available.
By \eqref{e1}, for $n\in [X, 2X)$, we have
\begin{align*}
\log k&> \log \left(0.8 \alpha^\alpha (\log X)^\alpha\right)\ge \frac 12 \alpha \log\log (3X)\\
&\ge \frac 12 \alpha \log\log (n+a_t)>\frac12\alpha\log\log n.
\end{align*}
Noting that
\begin{align*}
X\exp (-C_0 (\log X)^\alpha )&=(x\log\log x)^{1/\alpha} \exp (-C_0
\alpha^{-\alpha} (\log x+\log\log\log x)^\alpha )\\
&>x^{1/\alpha}\exp \left(-C_0
2^{\alpha}\alpha^{-\alpha} (\log x)^\alpha \right),
\end{align*}
thus there are at least
$x^{1/\alpha}\exp \left(-C_0
2^{\alpha}\alpha^{-\alpha} (\log x)^\alpha \right)$ of
integers $$n\in \left[(x\log\log x)^{1/\alpha}, 2(x\log\log x)^{1/\alpha} \right)$$ such that at least $\frac12C_0^{-1}\alpha\log\log n$ of the representations
$$\ell_i(n+a_t)=p+a_i \quad (p\in \mathbb{P},1\leqslant i\leqslant t)$$
are available. Taking $\widetilde{n}=n+a_t$, then there are at least
$x^{1/\alpha}\exp \left(-\widetilde{C_0}
2^{\alpha}\alpha^{-\alpha} (\log x)^\alpha \right)$ of
integers $$\widetilde{n}\in [(x\log\log x)^{1/\alpha}, 2(x\log\log x)^{1/\alpha} )$$such that at least $\frac12C_0^{-1}\alpha\log\log n$ of the representations
$$\ell_i\widetilde{n}=p+a_i \quad (p\in \mathbb{P},1\leqslant i\leqslant t)$$
hold, where $\widetilde{C_0}$ is a constant slightly larger than $C_0$.

This completes the proof of Theorem \ref{thm1}.
\end{proof}

\begin{proof}[Proof of Corollary \ref{cor1}] For a sufficiently
large $m$ with $\mathscr{A}(m)>(\log m)^\alpha $ and $\ell_i<0.9\log\log m$. We take
$\alpha_1 =\min \{ \alpha , \frac 15\} $.  Then
$\mathscr{A}(m)>(\log m)^{\alpha_1} $. By Theorem \ref{thm1},
there is  an integer $n\in \left[(m\log\log m)^{1/\alpha_1}, 2(m\log\log m)^{1/\alpha_1}
\right)$ such that \begin{equation}\label{e2} f_{\mathscr{A},\mathscr{L}}(n)\ge c'
\log\log n.\end{equation}
 Since there are infinitely many positive
integers $m$ such that $\mathscr{A}(m)>(\log m)^\alpha $, it
follows that
$$\limsup_{n\to \infty } \frac{f_{\mathscr{A},\mathscr{L}}(n)}{\log\log n}\ge
c'>0.$$

This completes the proof of Corollary \ref{cor1}.
\end{proof}

\begin{proof}[Proof of Corollary \ref{cor1a}] Take $\ell_i=1$ for all $i$. Let
$$\alpha_1 =\min\left\{ \alpha , \frac 15 \right\} .$$
For a sufficiently large $x$, let $X\log\log X=(x/2)^{\alpha_1} $. Then
$\mathscr{A}(X)>(\log X)^{\alpha } \ge (\log X)^{\alpha_1}$. By
Theorem \ref{thm1}, there are at least $X^{1/\alpha_1}\exp (-c
(\log X)^{\alpha_1} )$
 integers $$n\in \left[(X\log\log X)^{1/{\alpha_1}}, 2(X\log\log X)^{1/{\alpha_1}} \right)$$ with
$$f_{\mathscr{A}}(n)\geqslant 2c'\log\log n ,$$
where $c$ and $c'$ are two positive constant depending only on $\alpha_1 $.
Noting that $$(X\log\log X)^{1/\alpha_1}=\frac 12 x$$ and
\begin{eqnarray*}X^{1/\alpha_1}\exp (-c (\log X)^{\alpha_1} ) &\ge&\frac x2 (\log\log X)^{-1/\alpha_1}\exp
(-c {\alpha_1}^{\alpha_1} (\log x-\log 2)^{\alpha_1})\\
&\ge & x \exp (-c (\log x)^{\alpha_1} )\\
&\ge & x \exp (-c (\log x)^{\alpha} )\end{eqnarray*} for all
sufficiently large $x$. So there are at least $x\exp (-c (\log
x)^\alpha )$ integers $n\in [x/2, x)$ with
$$f_{\mathscr{A}}(n)\ge 2c'\log\log n\ge 2c'\log\log \frac x2>c'\log\log
x.$$

This completes the proof of Corollary \ref{cor1a}.
\end{proof}

\begin{proof}[Proof of Corollary \ref{cor2}] Let
$a_i=2^{i^\ell}$. Then $$\mathscr{A}(x)\ge (\log x/\log
2)^{1/\ell}
>(\log x)^{1/\ell}.$$
Now Corollary \ref{cor2} follows from Corollary \ref{cor1a}.
\end{proof}

\section{Further refinements}\label{111}
Maybe we should mention that the results obtained previously in this article can be strengthen slightly by requiring the integers $n$ chosen from a relatively sparse set. In the following, we sketch some adjustments of the previous proofs to make integers $n$ restricted in the Beatty sequence $\left\{\lfloor \beta k+\gamma\rfloor\right\}_{k=1}^{\infty}$, where $\beta>1$ is an irrational real number  and $\gamma$ a real number. Certainly, the adjustments are focused on the verifications of {\bf Hypothesis 1} and these can be pushed parallel with the domain $0<\theta<1/2$. In another word, Proposition \ref{Pro} can be established for $$\mathcal{A}=\left\{\lfloor \beta k+\gamma\rfloor\right\}_{k=1}^{\infty}.$$

The type $\tau$ of $\beta$ is defined by
$$\tau=\sup\left\{\varrho\in \mathbb{R}:\liminf_{n\to \infty}n^{\varrho}\|\beta n\|=0 \right\},$$
where $\|y\|$ denotes the the distance from the
real number $y$ to the nearest integer, i.e.,
$$\|y\|=\min_{n\in \mathbb{Z}}|y-n|.$$
Let the continued fraction expansion of $\beta$ be $[a_0, a_1, \dots ]$ and let $[a_0, a_1, \dots , a_n]=p_n/q_n$ with $q_n\geqslant 1$ and $(p_n, q_n)=1$.
The number $\beta$ is said to have bounded partial quotients if there is an absolute constant $K$ such
that $a_i\leqslant K$ for any $a_i$ in its continued fraction.

\begin{lemma}\cite[Theorem 182]{Hardy}\label{continuedfraction} For any integer $q$ with $1\leqslant q\leqslant q_n$, we have $\| q_n\beta \| \leqslant \| q\beta \|$.
\end{lemma}

\begin{lemma}\label{boundpq} If $\beta$ has bounded partial quotients, then the type of $\beta $ is $1$.
\end{lemma}

\begin{proof} Since $\beta$ has bounded partial quotients, there is an absolute constant $K$ such
that $a_i\leqslant K$ for all $i\ge 0$. Let $\beta =[a_0, a_1, \dots , a_n, \beta_n ]$. Then $\beta_n< K+1$ and
$$\beta -\frac{p_n}{q_n}=\frac{\beta_n p_n+p_{n-1}}{\beta_n q_n+q_{n-1}}-\frac{p_n}{q_n}=\frac{(-1)^n}{q_n (\beta_n q_n+q_{n-1})}.$$
Thus,
$$q_n\| \beta q_n \| =q_n |\beta q_n -p_n | =\frac {q_n}{\beta_n q_n+q_{n-1}}>\frac {q_n}{(K+1) q_n+q_{n-1}}\geqslant \frac 1{K+2}.$$
Clearly, $q_n=a_nq_{n-1}+q_{n-2}\leqslant (K+1)q_{n-1}$.
It follows from Lemma \ref{continuedfraction} that for $q_{n-1}<q\le q_n$, 
$$ q\| \beta q \| > q_{n-1} \| \beta q_n \|\geqslant \frac 1{(K+1)(K+2)}.$$
Hence
$$\liminf_{q\to +\infty} q\| \beta q \| \geqslant \frac 1{(K+1)(K+2)}.$$
For any $\varepsilon >0$, by 
$$q_n\| \beta q_n \| =|\beta q_n -p_n | =\frac {q_n}{\beta_n q_n+q_{n-1}}<1,$$
we have 
$$\liminf_{q\to +\infty} q^{1-\varepsilon }\| \beta q \| \le \liminf_{n\to +\infty} q_n^{1-\varepsilon }\| \beta q_n \| =0.$$
Therefore, the type of $\beta $ is $1$.
\end{proof}

The following standard results illustrated that the elements are well--distributed in $\left\{\lfloor \beta k+\gamma\rfloor\right\}_{k=1}^{\infty}$.

\begin{lemma}[\cite{BG}, Theorem]\label{le5-1}
Let $\beta>1$ be irrational with bounded partial quotients and $0\leqslant\gamma<\beta$. Then for integers $a,d$ with $d \geqslant 2$ and $0\leqslant a<d$, we have
$$\sum_{\substack{n\leqslant x\\\lfloor\beta n+\gamma\rfloor\equiv a\pmod{d}}}1=\frac{x}{d}+O\left(d\log^3x\right),$$
where the implied constant depends at most on $\beta$ and $\gamma$.
\end{lemma}
\begin{lemma}[\cite{BS}, Theorem 5.1]\label{le5-2}
Let $\beta$ and $\gamma$ be fixed real numbers with $\beta$ irrational, positive and of finite type. Then there is a constant $\kappa$ (depending on $\beta$) such that for all integers $0\leqslant a\leqslant q\leqslant N^{\kappa}$ with $(a,q)=1$, we have
$$\sum_{\substack{n\leqslant N}}\Lambda\left(q\lfloor \beta n+\gamma\rfloor+a\right)=\beta^{-1}\sum_{\substack{m\leqslant\lfloor \beta N+\gamma\rfloor}}\Lambda\left(qm+a\right)+O\left(N^{1-\kappa}\right),$$
where the implied constant depends only on $\beta$ and $\gamma$.
\end{lemma}
\begin{lemma}[\cite{BS}, Theorem 5.4]\label{le5-3}
Let $\beta$ and $\gamma$ be fixed real numbers with $\beta$ irrational, positive and of finite type. Then there is a constant $\kappa$ (depending on $\beta$) such that for all integers $0\leqslant a\leqslant q\leqslant N^{\kappa}$ with $(a,q)=1$, we have
$$\sum_{\substack{n\leqslant N\\\lfloor \beta n+\gamma\rfloor\equiv a\pmod{q}}}\Lambda\left(\lfloor \beta n+\gamma\rfloor\right)=\beta^{-1}\sum_{\substack{m\leqslant\lfloor \beta N+\gamma\rfloor\\ m\equiv a\pmod{q}}}\Lambda\left(m\right)+O\left(N^{1-\kappa}\right),$$
where the implied constant depends only on $\beta$ and $\gamma$.
\end{lemma}

Conditions (1) and (3) in {\bf Hypothesis 1} can be verified by Lemma \ref{le5-1}. Unluckily, when we need to verify condition (2) of {\bf Hypothesis 1}, both Lemma \ref{le5-2} and Lemma \ref{le5-3} can not be applicable.
However, we can establish a variant of Lemma \ref{le5-2} and Lemma \ref{le5-3} to break this barrier.

\begin{theorem}\label{th5}
Let $\beta>1$ and $\gamma$ be fixed real numbers with $\beta$ irrational, positive and of finite type. Then there is a constant $\kappa_1$ (depending on $\beta$) such that for all integers $0\leqslant a_j\leqslant q_j\leqslant N^{\kappa_1}~~(j=1,2)$ with $(a_2,q_2)=1$ and $(q_2a_1+a_2,q_1)=1$, we have
$$\sum_{\substack{n\leqslant N\\ \lfloor \beta n+\gamma\rfloor\equiv a_1\pmod{q_1} }}\Lambda\left(q_2\lfloor \beta n+\gamma\rfloor+a_2\right)=\beta^{-1}\sum_{\substack{m\leqslant\lfloor \beta N+\gamma\rfloor\\m\equiv a_1\pmod{q_1}}}\Lambda\left(q_2m+a_2\right)+O\left(N^{1-\kappa_1}\right),$$
where the implied constant depends only on $\beta$ and $\gamma$.
\end{theorem}

The proof of Theorem \ref{th5} is actually a trivial adjustment of the proof of  \cite[Theorem 5.1]{BS} (i.e., Lemma \ref{le5-2}) by making the following replacement in it.
$$\sum_{\substack{m\leqslant M\\m\equiv a_1\pmod{q_1}}}\!\!\!\Lambda(q_2m+a_2)e(\beta km)\xlongequal{n=q_2m+a_2}e\left(-\frac{\beta ka_2}{q_2}\right)\!\!\!\!\sum_{\substack{n\leqslant q_2M+a_2\\n\equiv q_2a_1+a_2\pmod{q_1q_2}}}\!\!\!\Lambda(n)e\left(\frac{\beta kn}{q_2}\right).$$
We remark that in this equality the additional condition $(q_2a_1+a_2,q_1)=1$ appeared in Theorem \ref{th5} is necessary, which shall arise naturally in the verification of condition (2) in {\bf Hypothesis 1}.  Summarizing the above analysis and statements, it can be concluded as the following theorem.
\begin{theorem}\label{outline}
Let $\beta>1$ and $0\leqslant\gamma<\beta$ be fixed real numbers such that $\beta$ is irrational with bounded partial quotients.  For $\alpha >0$, let $$\mathscr{A}=\{ a_1<a_2<a_3<\cdots\}$$ and $$\mathscr{L}=\{ \ell_1, \ell_2, \ell_3,\cdots\} \quad \text{(not~necessarily~different)}$$ be two sequences of positive
integers with $\mathscr{A}(m)>(\log m)^\alpha $ for infinitely
many positive integers $m$ and $\ell_m<0.9\log\log m$ for sufficiently large integers $m$. Then we have
$$\limsup_{\substack{n\to \infty \\ n\in \mathcal{A}}} \frac{f_{\mathscr{A},\mathscr{L}}(n)}{\log\log n}>0$$
where $\mathcal{A}=\left\{\lfloor \beta k+\gamma\rfloor\right\}_{k=1}^{\infty}$ is the Beatty sequence.
\end{theorem}

Perhaps, some further refinements of our previous results can be taken in the scene of the  Piatetski--Shapiro  sequences \cite{PSh} $\lfloor n^c\rfloor$ for some $c$ slightly larger than $1$ due to the well--distributed property of its elements. We shall admit that we do not proceed this work. We believe that an elaborate treatment of the exponential sums involving the Piatetski--Shapiro sequences $\lfloor n^c\rfloor$ to make the constant $c$ as large as possible would be of some interest.

\section{Final remarks}

There is an outstanding conjecture named {\it Prime k-tuples conjecture} of Dickson which states as following.

\begin{conjecture}[weak form] Let $a_1,a_2,...,a_k$ be an admissible set, then there are infinitely many integers $n$ such that all of $n+a_1,n+a_2,...,n+a_k$ are primes.
\end{conjecture}

Based on the conjecture of Dickson and recent progress towards this direction, especially the result of Maynard \cite[Theorem 3.1]{Ma2}, we can make the following reasonable conjecture. It can be viewed as an extension of the Dickson conjecture.

\begin{conjecture}\label{conj1} For any arbitrarily small constant $\delta>0$, there exists an integer $x_0=x_0 (\delta )$ such that
for every $x\ge x_0$, if $a_1,a_2,...,a_k$ is an admissible set
with $k\le (\log x)^{1-\delta }$ and $a_i\leqslant x^{1-\delta }$,
then there is a positive integer $n\le x$
 such that all of the $k$ shifts $n+a_1,n+a_2,...,n+a_k$
are primes.
\end{conjecture}

Replacing Theorem \ref{M} by Conjecture \ref{conj1}, we shall have
$$\limsup_{n\rightarrow\infty}\frac{f_{\mathscr{A}}(n)}{(\log n)^{1-\delta}}>0.$$
Certainly, it covers the result $$\limsup_{n\rightarrow\infty}\frac{f(n)}{(\log n)^{1-\delta}}>0.$$
However, the decision of whether $f(n)=o(\log n)$ seems to be far beyond the present mathematics, even under the strong assumptions. We end this article by quoting the original words of Erd\H os. `{\it The conjecture $f(n)=o(\log n)$ if true is probably rather deep}'.

\section*{Funding}
The first author is supported by the National Natural Science
Foundation of China, Grant No. 12171243. The second author is supported by National Natural Science Foundation of China under Grant No. 12201544, Natural Science Foundation of Jiangsu Province, China, Grant No. BK20210784, China Postdoctoral Science Foundation, Grant No. 2022M710121, the foundations of the projects "Jiangsu Provincial Double--Innovation Doctor Program'', Grant No. JSSCBS20211023 and "Golden  Phoenix of the Green City--Yang Zhou'' to excellent PhD, Grant No. YZLYJF2020PHD051.

\begin{acknowledgment}
The authors would like to thank the anonymous referee for his/her helpful comments which improve the quality of their article greatly.
\end{acknowledgment}


\begin{thebibliography}{KMP}
\bibitem{BS} W. Banks, I. Shparlinski, {\it Prime numbers with Beatty sequences}, Colloq. Math., {\bf115} (2009), 147--157.
\bibitem{BG} A. V. Begunts, D. V. Goryashin, {\it On the values of Beatty sequence in an arithmetic progression}, Chebyshevskii Sb., {\bf 21} (2020), 364--367.
\bibitem{Ch} Y.--G. Chen, {\it Romanoff theorem in a sparse set}, Sci. China Math., {\bf 53} (2010), 2195--2202.
\bibitem{CD} Y.--G. Chen, Y. Ding, {\it On a conjecture of Erd\H os}, C. R. Math. Acad. Sci. Paris, {\bf 360} (2022), 971--974.
\bibitem{CS} Y.--G. Chen, X.--G. Sun, {\it On Romanoff's constant}, J. Number Theory, {\bf106} (2004), 275--284.
\bibitem{va} J. G. van der Corput, {\it On de Polignac's conjecture}, Simon Stevin, {\bf27} (1950), 99--105.
\bibitem{Da} H. Davenport, {\it Multiplicative Number Theory}, Second edition, Graduate Texts in Mathematics 74, Springer-Verlag, New York, 1980.
\bibitem{DDD} G. M. Del Corso, I. Del Corso, R. Dvornicich, F. Romani, {\it On computing the density of integers of the form $2^n+p$}, Math. Comp., {\bf89} (2020), 2365--2386.
\bibitem{Di} Y. Ding, {\it Extending an Erd\H{o}s result on a Romanov type problem}, Arch. Math. (Basel), {\bf 118} (2022), 587--592.
\bibitem{DZ} Y. Ding, G.--L. Zhou, {\it Some application of the admissible set}, Amer. Math. Monthly, https://doi.org/10.1080/00029890.2023.2184621.
\bibitem{Els} C. Elsholtz, {\it Upper bounds for prime $k$--tuples of size $\log N$ and oscillations,} Arch. Math. (Basel), {\bf 82} (2004), 33--39.
\bibitem{ES} C. Elsholtz, J.--C. Schlage--Puchta, {\it On Romanov's constant}, Math. Z., {\bf288} (2018), 713--724.
\bibitem{Er1} P. Erd\H{o}s, {\it On the integers of the form $2^k+p$ and some related problems}, Summa Brasil. Math., {\bf 2} (1950), 113--123.
\bibitem{FI} J. B. Friedlander, H. Iwaniec, {\it Opera de Cribro}, Colloq. Publications, vol. 57, Amer. Math. Soc., Providence, Rhode Island, 2010.
\bibitem{HL} L. Habsieger, X. Roblot, {\it On integers of the form $p+2^k$}, Acta Arith., {\bf 122} (2006), 45--50.
\bibitem{HS} L. Habsieger, J. Sivak-Fischler,
{\it An effective version of the Bombieri-Vinogradov theorem, and applications to Chen's theorem and to sums of primes and powers of two},
Arch. Math. (Basel), {\bf95} (2010), 557--566.
\bibitem{Hardy} G. H. Hardy, E. M. Wright, {\it An intrduction to the theory of numbers}, Oxford Univ. Press, 1979.
\bibitem{Ho} C. Hooley, {\it Applications of sieve methods to the theory of numbers}. Cambridge Tracts in Mathematics, No. 70. Cambridge University Press, Cambridge-New York-Melbourne, 1976.
\bibitem{Lu} Y. Lu, {\it On almost primes of the form $p+2^l$}, J. Number Theory, {\bf 204} (2019), 264--278.
\bibitem{Lv} G. L\"{u}, {\it On Romanoff's constant and its generalized problem}, Adv. Math., (China) {\bf36} (2007), 94--100.
\bibitem{Ma} J. Maynard, {\it Small gaps between primes}, Ann. of Math., {\bf 181} (2015), 383--413.
\bibitem{Ma2} J. Maynard, {\it Dense clusters of primes in subsets}, Compos. Math., {\bf152} (2016), 1517--1554.
\bibitem{Mo} H. L. Montgomery, {\it A note on the large sieve}, J. London Math. Soc., {\bf43} (1968), 93--98.
\bibitem{Na} W. Narkiewicz, {\it On a conjecture of Erd\H{o}s}, Colloq. Math., {\bf 37} (1977), 313--315.
\bibitem{de1} A. de Polignac, {\it Six propositions arithmologiques d\'{e}duites du crible d'Eratosth\'{e}ne}, Nouv. Ann. Math., {\bf8} (1849), 423--429.
\bibitem{Pi} J. Pintz, {\it A note on Romanoff's constant}, Acta Math. Hung., {\bf112} (2006), 1--14.
\bibitem{de2} A. de Polignac, {\it Recherches nouvelles sur ies nombres primiers}. C. R. Acad. Sci. Paris, {\bf 29} (1849), 738--739.
\bibitem{Ta} D. H. J. Polymath, {\it Variants of the Selberg sieve, and bounded intervals containing many primes}, Res. Math. Sci.,
{\bf 1} (2014), 1--83.
\bibitem{PSh} I. I. Pyatecki\v{i}--\v{S}apiro, {\it On the distribution of prime numbers in sequences of the form $\left[f(n)\right]$,} Mat. Sbornik N.S. {\bf 75} (1953), 559--566.
\bibitem{Ro} N. P. Romanoff, {\it \"Uber einige S\"atze der additiven
Zahlentheorie}, Math. Ann. 57 (1934), 668--678.
\bibitem{Va} R. C. Vaughan, {\it Some applications of Montgomery's sieve}, J. Number Theory, {\bf 5} (1973), 64--79.
\end{thebibliography}
\end{document}